\documentclass[11pt]{article}

\usepackage[a4paper,margin=1in]{geometry}
\usepackage[T1]{fontenc}
\usepackage[utf8]{inputenc}
\usepackage{lmodern}
\usepackage{microtype}
\usepackage{amsmath,amssymb,amsthm,mathtools}
\usepackage{booktabs}
\usepackage{array}
\usepackage{tabularx}
\usepackage{xurl}
\usepackage{placeins}
\usepackage{hyperref}
\usepackage{authblk}

\allowdisplaybreaks
\numberwithin{equation}{section}

\hypersetup{
  unicode=true,
  colorlinks=true,
  linkcolor=blue,
  citecolor=blue,
  urlcolor=blue,
  pdftitle={A complete classification for an Erdos unimodality problem on prime divisors},
  pdfauthor={Shouqiao Wang and Davide Crapis}
}
\theoremstyle{plain}
\newtheorem{theorem}{Theorem}[section]
\newtheorem{lemma}[theorem]{Lemma}
\newtheorem{proposition}[theorem]{Proposition}
\newtheorem{corollary}[theorem]{Corollary}

\theoremstyle{definition}
\newtheorem{certificate}[theorem]{Certificate}

\theoremstyle{remark}

\newcolumntype{Y}{>{\raggedright\arraybackslash}X}

\newcommand{\e}{\mathrm e}
\newcommand{\eps}{\varepsilon}

\title{
A Complete Answer to Erdős Problem 690\\
{\large Discovered by the Multiscalar Fields System}
}

\author{Shouqiao Wang}
\author{Davide Crapis}
\affil{Multiscalar Intelligence}

\date{}

\begin{document}

\maketitle

\begin{abstract}
Let \(d_k(p)\) denote the natural density of positive integers whose \(k\)-th smallest prime divisor is \(p\).  Erd\H{o}s asked whether, for each fixed \(k\), the sequence \(p\mapsto d_k(p)\) is unimodal as \(p\) ranges over the primes.  Cambie proved that unimodality holds for \(1\le k\le3\) and verified non-unimodality for \(4\le k\le20\).  We prove that \(p\mapsto d_k(p)\) is not unimodal for every \(k\ge4\), completing the classification. An exact first-difference criterion reduces the problem to comparing a symmetric-polynomial ratio with prime gaps.  Explicit estimates for prime-counting functions, certified finite computations, one certified large prime gap, one certified twin prime, and a uniform Chinese-remainder construction then produce, for every \(k\ge4\), a strict descent followed by a later strict ascent.
\end{abstract}

\section{Introduction}

We developed the Multiscalar Fields System to automate large-scale mathematical exploration, theorem discovery and verification. In this paper, we present a proof of Erd\H{o}s Problem 690 discovered by Fields with limited human interaction. Starting from an initial problem specification, Fields explored the space of candidate constructions and arguments, evaluated their mathematical viability, and iteratively refined them until they converged to the formal math proof presented here. Human involvement was limited to formulating the problem, auditing the generated proof and verifying the numerical computations.

\medskip
\noindent \textbf{Problem.}
Let \(d_k(p)\) be the natural density of positive integers \(n\) for which \(p\) is the \(k\)-th smallest prime divisor of \(n\).  Erd\H{o}s asked whether, for fixed \(k\ge1\), the sequence
\[
        p\longmapsto d_k(p),
\]
indexed by the primes in increasing order, is unimodal \cite{Erdos1979}; see also the modern formulation in \cite{Bloom690}.  Here an infinite sequence is called unimodal if it is nondecreasing up to some index and nonincreasing after that index.

Cambie \cite{Cambie2025} proved that \(d_k(p)\) is unimodal for \(1\le k\le3\) and verified non-unimodality in the initial range \(4\le k\le20\).  The present paper gives a proof of non-unimodality for every \(k\ge4\), including an independent verification of the finite initial range.

\begin{theorem}\label{thm:main}
For every integer
\[
        k\ge 4,
\]
the sequence
\[
        p\longmapsto d_k(p),
\]
where \(p\) runs through the primes in increasing order, is not unimodal.
\end{theorem}

Together with Cambie's result for \(1\le k\le3\), Theorem \ref{thm:main} gives the following complete answer to the Erd\H{o}s question.

\begin{corollary}\label{cor:classification}
The sequence \(p\mapsto d_k(p)\) is unimodal for \(1\le k\le3\), and is not unimodal for every \(k\ge4\).
\end{corollary}

We shall repeatedly use the following elementary observation.  If a sequence has a strict descent \(a_i>a_{i+1}\) and, at a later adjacent pair, a strict ascent \(a_j<a_{j+1}\) with \(i<j\), then it is not unimodal.  Indeed, a mode would have to lie at or before \(i\), because of the descent, and at or after \(j+1\), because of the ascent.

The proof has two conceptual parts.  First we derive an exact threshold criterion for the sign of a first difference of \(d_k(p)\).  The criterion compares a ratio \(R_r\) of elementary symmetric functions with the adjacent prime gap.  Second we produce, for each \(k\), an earlier prime gap that is large enough to force a strict descent and a later prime gap that is small enough to force a strict ascent.  The finite range uses certified computations and two published prime-record certificates.  The infinite tail uses a Chinese-remainder construction to manufacture a long block of composite integers, followed by an average-gap argument that finds a smaller later gap.

The organization is as follows.  Section \ref{sec:notation} fixes notation.  Section \ref{sec:structural} proves the two structural comparison lemmas.  Section \ref{sec:inputs} records the analytic estimates, numerical constants, and external certificates.  Section \ref{sec:finite} handles all \(4\le k\le 8{,}600{,}001\).  Section \ref{sec:tail} gives a uniform construction for every remaining \(k\).  Appendix \ref{app:verification} describes the companion numerical verifier.

\section{Notation}\label{sec:notation}

We index the prime sequence:
\[
        p_1=2,\quad p_2=3,\quad p_3=5,\quad \ldots .
\]
Let
\[
        g_i:=p_{i+1}-p_i
\]
for \(i\ge1\).
For an integer \(m\ge 0\) and \(i\ge0\), let \(\delta_m(i)\) be the natural density of integers \(n\) which are divisible by exactly \(m\) primes from the set
\[
        \{p_1,\ldots,p_i\}.
\]
For \(i=0\) this set is empty, so
\[
        \delta_0(0)=1,\qquad \delta_m(0)=0\quad(m>0).
\]

The elementary independence of divisibility by distinct primes gives, for \(i\ge1\) and \(m\ge0\), with \(\delta_{-1}(i-1)=0\),
\[
        \delta_m(i)=\left(1-\frac1{p_i}\right)\delta_m(i-1)
              +\frac1{p_i}\delta_{m-1}(i-1).
\]
Also, for \(i\ge1\),
\[
        d_{m+1}(p_i)=\frac{\delta_m(i-1)}{p_i}.
\]

For \(r\ge 1\) define, whenever the denominator is positive,
\[
        R_r(i):=\frac{\delta_{r-1}(i)}{\delta_r(i)}.
\]

\section{Structural Lemmas}
\label{sec:structural}

\begin{lemma}[The threshold criterion]\label{lem:threshold}
Let \(r\ge 1\) and \(i\ge1\).  Suppose \(\delta_r(i-1)>0\).  Then
\[
        d_{r+1}(p_{i+1})>d_{r+1}(p_i)
        \quad\Longleftrightarrow\quad
        R_r(i-1)>g_i+1.
\]
Consequently, if \(R_r(i-1)<g_i+1\), then
\[
        d_{r+1}(p_{i+1})<d_{r+1}(p_i).
\]
\end{lemma}

\begin{proof}
By the recurrence for \(\delta_r(i)\),
\[
\begin{aligned}
d_{r+1}(p_{i+1})-d_{r+1}(p_i)
&=\frac{\delta_r(i)}{p_{i+1}}-\frac{\delta_r(i-1)}{p_i}.
\end{aligned}
\]
Multiplying by the positive number \(p_i p_{i+1}\), the sign is the sign of
\[
        p_i\delta_r(i)-p_{i+1}\delta_r(i-1).
\]
Now
\[
        p_i\delta_r(i)
        =(p_i-1)\delta_r(i-1)+\delta_{r-1}(i-1),
\]
and since \(p_{i+1}=p_i+g_i\), we get
\[
\begin{aligned}
p_i\delta_r(i)-p_{i+1}\delta_r(i-1)
&=\delta_{r-1}(i-1)+(p_i-1-p_{i+1})\delta_r(i-1)\\
&=\delta_{r-1}(i-1)-(g_i+1)\delta_r(i-1).
\end{aligned}
\]
Since \(\delta_r(i-1)>0\), the asserted equivalence follows after division by
\(\delta_r(i-1)\).
\end{proof}

For the second lemma set
\[
        w_j:=\frac1{p_j-1}\quad(j\ge1),
        \qquad
        E_m(i):=e_m(w_1,\ldots,w_i),
\]
where \(e_m\) is the \(m\)-th elementary symmetric polynomial.  Also put
\[
        A(y):=\sum_{p\le y}\frac1{p-1},
        \qquad
        A(y^-):=\sum_{p<y}\frac1{p-1},
        \qquad
        W_N:=\sum_{j=1}^{N}\frac1{p_j-1}.
\]
The sum defining \(W_N\) is empty when \(N=0\).  Thus \(W_N\) is the sum of the first \(N\) weights, and for every \(N\ge1\),
\[
        W_N=A(p_N).
\]

\begin{lemma}[Symmetric-polynomial bounds]\label{lem:Rbounds}
Let \(r\ge 1\), and suppose that \(i\ge r\).  Then \(\delta_r(i)>0\), and
\[
        \frac r{A(p_i)}
        \le
        R_r(i)
        \le
        \frac r{A(p_i)-W_{r-1}} .
\]
The denominator in the upper bound is positive because \(i\ge r\).
\end{lemma}

\begin{proof}
For a fixed subset \(S\subseteq\{1,\ldots,i\}\), the contribution of the event that exactly the primes \(p_j\) with \(j\in S\) divide \(n\) is
\[
        \prod_{j\in S}\frac1{p_j}
        \prod_{j\notin S}\left(1-\frac1{p_j}\right)
        =
        \left(\prod_{j=1}^i\frac{p_j-1}{p_j}\right)
        \prod_{j\in S}\frac1{p_j-1}.
\]
Therefore
\[
        \delta_m(i)
        =
        \left(\prod_{j=1}^i\frac{p_j-1}{p_j}\right)E_m(i),
\]
and hence
\[
        R_r(i)=\frac{E_{r-1}(i)}{E_r(i)}.
\]

Let
\[
        \alpha_i:=A(p_i)=\sum_{j=1}^i w_j .
\]
Expanding \(\alpha_iE_{r-1}(i)\), we obtain
\[
        \alpha_iE_{r-1}(i)=rE_r(i)+\Sigma_i,
\]
where
\[
        \Sigma_i
        :=
        \sum_{\substack{S\subseteq\{1,\ldots,i\}\\ |S|=r-1}}
        \left(\sum_{s\in S}w_s\right)\prod_{s\in S}w_s
        \ge 0.
\]
Since \(w_1\ge w_2\ge w_3\ge\cdots\), every set \(S\) of size \(r-1\) satisfies
\[
        \sum_{s\in S}w_s\le w_1+\cdots+w_{r-1}=W_{r-1},
\]
with the right-hand side interpreted as the empty sum when \(r=1\).
Thus
\[
        0\le \Sigma_i\le W_{r-1}E_{r-1}(i).
\]
Dividing the identity
\[
        \alpha_iE_{r-1}(i)=rE_r(i)+\Sigma_i
\]
by \(E_{r-1}(i)>0\), and writing
\[
        s_i:=\frac{\Sigma_i}{E_{r-1}(i)},
\]
we get
\[
        0\le s_i\le W_{r-1},
        \qquad
        R_r(i)=\frac r{A(p_i)-s_i}.
\]
The lower bound follows from \(s_i\ge 0\).  Since \(i\ge r\),
\[
        A(p_i)-W_{r-1}
        =\sum_{j=r}^{i}w_j
        \ge w_r>0.
\]
Together with \(s_i\le W_{r-1}\), this gives
\[
        A(p_i)-s_i\ge A(p_i)-W_{r-1}>0,
\]
and hence the stated upper bound.
\end{proof}

\section{Explicit Estimates and Certificates}
\label{sec:inputs}

We shall use the following known explicit estimates.

\begin{lemma}[Explicit estimates for primes]\label{lem:dusart}
The following estimates hold.
\begin{enumerate}
\item For every \(x\ge 3275\), there is a prime \(\ell\) such that
\begin{equation}\label{eq:short-interval-prime}
x<\ell\le x\left(1+\frac1{2\log^2 x}\right).
\end{equation}
\item For \(x>0\),
\begin{equation}\label{eq:theta-upper}
\vartheta(x)-x<\frac{x}{36260}.
\end{equation}
For \(x>2\),
\begin{equation}\label{eq:theta-lower}
\vartheta(x)>x\left(1-\frac{1.2323}{\log x}\right).
\end{equation}
\item For \(x>5393\),
\begin{equation}\label{eq:pi-lower}
\pi(x)\ge \frac{x}{\log x-1}.
\end{equation}
For \(x>60184\),
\begin{equation}\label{eq:pi-upper}
\pi(x)\le \frac{x}{\log x-1.1}.
\end{equation}
\item For the primes \(p_n\) indexed as in Section \ref{sec:notation}, if \(n>688383\),
\begin{equation}\label{eq:nth-prime-upper}
p_n
        \le
        n\left(
        \log n+\log\log n-1+
        \frac{\log\log n-2}{\log n}
        \right).
\end{equation}
\end{enumerate}
\end{lemma}

The short-interval estimate \eqref{eq:short-interval-prime} is due to Dusart's thesis and is quoted, for example, by Axler \cite{Axler2018}.  The estimates \eqref{eq:theta-upper}, \eqref{eq:theta-lower}, \eqref{eq:pi-lower}, \eqref{eq:pi-upper}, and \eqref{eq:nth-prime-upper} are Dusart's explicit estimates \cite{Dusart2010}.

We also need explicit estimates for \(A(y)\).  Let
\[
        B:=\gamma+\sum_p\left(\log\left(1-\frac1p\right)+\frac1p\right),
\]
the Meissel-Mertens constant, and put
\[
        \eps(y):=\frac1{10\log^2 y}+\frac4{15\log^3 y}.
\]
Dusart's reciprocal-prime estimate gives
\begin{equation}\label{eq:reciprocal-prime-estimate}
-\eps(y)
        \le
        \sum_{p\le y}\frac1p-\log\log y-B
        \le
        \eps(y),
\end{equation}
where the lower bound holds for \(y>1\), and the upper bound holds for \(y>10372\).  This reciprocal-prime estimate is due to Dusart \cite{Dusart2010}.

\begin{certificate}[Numerical constants]\label{cert:constants}
We use the outward-rounded interval
\begin{equation}\label{eq:B-interval}
0.261497212847642<B<0.261497212847643,
\end{equation}
for the Meissel-Mertens constant.  This interval is implied by the tabulated value
\[
        B=0.2614972128476427837554268386\ldots,
\]
consistent with the high-precision computation methods of Languasco and Zaccagnini \cite{LanguascoZaccagnini2010} and with the tabulated decimal expansion \cite{OEISMertens}.

With
\[
        C:=\sum_p\frac1{p(p-1)},
\]
we also use
\begin{equation}\label{eq:C-interval}
0.773156636699192<C<0.773157136700943.
\end{equation}
For this second interval, let \(N=1{,}999{,}993\).  Since all summands are positive,
\[
        \sum_{p\le N}\frac1{p(p-1)}<C.
\]
Moreover,
\[
        C-\sum_{p\le N}\frac1{p(p-1)}
        \le
        \sum_{n\ge N+1}\frac1{n(n-1)}
        =
        \frac1N,
\]
where the last equality follows by telescoping.  The finite partial sum over \(p\le N\) is evaluated with outward rounding in the companion verifier; equivalently, it may be evaluated by exact rational arithmetic and compared with the displayed decimal endpoints by integer cross-multiplication.
\end{certificate}

Since
\[
        \frac1{p-1}=\frac1p+\frac1{p(p-1)},
\]
we have, for \(y>10372\),
\begin{equation}\label{eq:A-upper}
A(y)\le \log\log y+B+\eps(y)+C.
\end{equation}
In particular, using only \(C<1\), we also have
\begin{equation}\label{eq:A-upper-weak}
A(y)\le \log\log y+B+1+\eps(y).
\end{equation}
For \(y>1\), since the second summand is positive,
\begin{equation}\label{eq:A-lower-weak}
A(y)\ge \log\log y+B-\eps(y).
\end{equation}
For large finite inputs in Section \ref{sec:finite-large}, we shall also use the following sharper lower bound.  If \(y\ge 1{,}999{,}993\), then
\begin{equation}\label{eq:A-lower-sharp}
A(y)\ge \log\log y+B_- -\eps(y)+C_- -\frac1{y-1},
\end{equation}
where
\[
        B_-:=0.261497212847642,
        \qquad
        C_-:=0.773156636699192.
\]
Indeed, after subtracting the reciprocal-prime estimate from the identity
\(1/(p-1)=1/p+1/(p(p-1))\), the only point is to lower-bound
\(\sum_{p\le y}1/(p(p-1))\).  Since
\[
        \sum_{p>y}\frac1{p(p-1)}
        \le
        \sum_{n>y}\frac1{n(n-1)}
        \le \frac1{y-1},
\]
this partial sum is at least \(C_- -1/(y-1)\).

All finite decimal inequalities used below are certified by the companion verifier described in Appendix \ref{app:verification}.  The finite prime sums and the elementary symmetric sums are checked exactly, while logarithmic comparisons and the value of \(C\) are checked by rigorous outward-rounded interval arithmetic.

\subsection{Published prime-record certificates}

We use two certified prime-record inputs.

\begin{certificate}[A certified large prime gap \cite{PrimeRecordsGap,PrimeGapList}]\label{input:largegap}
Let \(m\#:=\prod_{p\le m}p\) denote the primorial, where the product is over primes.  Set
\[
        s_L:=587\cdot \frac{43103\#}{2310}-455704.
\]
Since \(2310=11\#\), this is the same number as \(587\cdot(43103\#/11\#)-455704\), the form used in the published record tables.  There is a certified prime gap of length
\[
        g_L=1{,}113{,}106
\]
from \(s_L\) to \(s_L+g_L\).  In particular, \(s_L\) and \(s_L+g_L\) are consecutive primes.  Both endpoints have \(18662\) decimal digits.
\end{certificate}

\begin{certificate}[A certified twin prime \cite{PrimePagesTwin,PrimePagesSpecificTwin}]\label{input:twin}
Let
\[
        s_T:=504983334^{8192}-504983334^{4096}-1.
\]
The published PrimePages record lists this number as a proven twin prime.  Thus \(s_T\) and \(s_T+2\) are primes, and \(s_T\) has \(71298\) decimal digits.
\end{certificate}

These two certificates are used only as published prime-record data.  A fully formal version may replace them by the corresponding primality certificates for the endpoints and compositeness certificates for the intervening integers.

\section{The Finite Range: \texorpdfstring{\(4\le k\le 8{,}600{,}001\)}{4 <= k <= 8,600,001}}
\label{sec:finite}

\subsection{The small cases: \texorpdfstring{\(4\le k\le20\)}{4 <= k <= 20}}

The following table is a certified finite calculation.  For a prime gap \(a\to b\), the notation
\[
        R_r(a^-)
\]
means \(R_r(i-1)\), where \(p_i=a\).  The table is obtained by computing the elementary symmetric sums \(E_m(i)\) exactly from the rational weights \(w_j=1/(p_j-1)\).  The displayed decimals are rounded outward.

\begin{table}[htbp]
\centering
\small
\caption{Exact certificates for the small cases. The first displayed gap in each row gives a strict descent, and the second, later gap gives a strict ascent.}
\label{tab:smallcases}
\setlength{\tabcolsep}{4pt}
\begin{tabularx}{\textwidth}{cYY}
\toprule
\(r\) & Descent certificate & Later ascent certificate\\
\midrule
3 & \(13\to17\), \(g=4\): \(R_3(13^-)<3.506<5\) & \(17\to19\), \(g=2\): \(R_3(17^-)>3.048>3\) \\
4 & \(23\to29\), \(g=6\): \(R_4(23^-)<4.759<7\) & \(29\to31\), \(g=2\): \(R_4(29^-)>4.371>3\) \\
5 & \(31\to37\), \(g=6\): \(R_5(31^-)<6.748<7\) & \(37\to41\), \(g=4\): \(R_5(37^-)>6.263>5\) \\
6 & \(73\to79\), \(g=6\): \(R_6(73^-)<6.437<7\) & \(79\to83\), \(g=4\): \(R_6(79^-)>6.282>5\) \\
7 & \(89\to97\), \(g=8\): \(R_7(89^-)<8.085<9\) & \(97\to101\), \(g=4\): \(R_7(97^-)>7.911>5\) \\
8 & \(113\to127\), \(g=14\): \(R_8(113^-)<9.303<15\) & \(127\to131\), \(g=4\): \(R_8(127^-)>9.145>5\) \\
9 & \(113\to127\), \(g=14\): \(R_9(113^-)<11.677<15\) & \(127\to131\), \(g=4\): \(R_9(127^-)>11.452>5\) \\
10 & \(113\to127\), \(g=14\): \(R_{10}(113^-)<14.414<15\) & \(127\to131\), \(g=4\): \(R_{10}(127^-)>14.101>5\) \\
11 & \(293\to307\), \(g=14\): \(R_{11}(293^-)<12.085<15\) & \(307\to311\), \(g=4\): \(R_{11}(307^-)>12.011>5\) \\
12 & \(293\to307\), \(g=14\): \(R_{12}(293^-)<14.050<15\) & \(307\to311\), \(g=4\): \(R_{12}(307^-)>13.959>5\) \\
13 & \(523\to541\), \(g=18\): \(R_{13}(523^-)<13.651<19\) & \(541\to547\), \(g=6\): \(R_{13}(541^-)>13.607>7\) \\
14 & \(523\to541\), \(g=18\): \(R_{14}(523^-)<15.415<19\) & \(541\to547\), \(g=6\): \(R_{14}(541^-)>15.364>7\) \\
15 & \(523\to541\), \(g=18\): \(R_{15}(523^-)<17.279<19\) & \(541\to547\), \(g=6\): \(R_{15}(541^-)>17.218>7\) \\
16 & \(887\to907\), \(g=20\): \(R_{16}(887^-)<16.752<21\) & \(907\to911\), \(g=4\): \(R_{16}(907^-)>16.721>5\) \\
17 & \(887\to907\), \(g=20\): \(R_{17}(887^-)<18.438<21\) & \(907\to911\), \(g=4\): \(R_{17}(907^-)>18.402>5\) \\
18 & \(887\to907\), \(g=20\): \(R_{18}(887^-)<20.191<21\) & \(907\to911\), \(g=4\): \(R_{18}(907^-)>20.151>5\) \\
19 & \(1129\to1151\), \(g=22\): \(R_{19}(1129^-)<20.742<23\) & \(1151\to1153\), \(g=2\): \(R_{19}(1151^-)>20.711>3\) \\
\bottomrule
\end{tabularx}
\end{table}
\FloatBarrier

For each row, Lemma \ref{lem:threshold} gives a strict descent at the first displayed gap and a strict ascent at the second displayed, later gap.  Hence \(d_{r+1}(p)\) is not unimodal for \(3\le r\le 19\), equivalently
\[
        4\le k\le 20.
\]

\subsection{The ranges \texorpdfstring{\(21\le k\le31\)}{21 <= k <= 31} and \texorpdfstring{\(32\le k\le48\)}{32 <= k <= 48}}

First consider
\[
        20\le r\le 30.
\]
The certified finite sums are
\[
        \sum_{p<15683}\frac1{p-1}
        >3.303755162423773,
\]
\[
        \sum_{p\le 15683}\frac1{p-1}
        <3.303818929800384,
\]
and
\[
        W_{29}<2.612642166507777.
\]
Since \(W_{r-1}\le W_{29}\), Lemma \ref{lem:Rbounds} gives at the gap \(15683\to15727\), whose length is \(44\),
\[
        R_r(15683^-)
        \le
        \frac{30}{3.303755162423773-2.612642166507777}
        <43.409<45=44+1.
\]
Thus Lemma \ref{lem:threshold} gives a strict descent at \(15683\to15727\).

At the next gap \(15727\to15731\), whose length is \(4\), Lemma \ref{lem:Rbounds} gives
\[
        R_r(15727^-)
        \ge
        \frac{20}{3.303818929800384}
        >6.053>5=4+1.
\]
Thus Lemma \ref{lem:threshold} gives a strict ascent at \(15727\to15731\).  Therefore \(d_{r+1}(p)\) is not unimodal for \(20\le r\le 30\), equivalently
\[
        21\le k\le 31.
\]

Next consider
\[
        31\le r\le 47.
\]
The certified finite sums are
\[
        \sum_{p<31397}\frac1{p-1}
        >3.372584257226677,
\]
\[
        \sum_{p\le31397}\frac1{p-1}
        <3.372616108417913,
\]
and
\[
        W_{46}<2.721441010945543.
\]
Since \(W_{r-1}\le W_{46}\), Lemma \ref{lem:Rbounds} gives at the gap \(31397\to31469\), whose length is \(72\),
\[
        R_r(31397^-)
        \le
        \frac{47}{3.372584257226677-2.721441010945543}
        <72.181<73=72+1.
\]
Thus there is a strict descent at \(31397\to31469\).

At the later gap \(31469\to31477\), whose length is \(8\), Lemma \ref{lem:Rbounds} gives
\[
        R_r(31469^-)
        \ge
        \frac{31}{3.372616108417913}
        >9.191>9=8+1.
\]
Thus there is a strict ascent at \(31469\to31477\).  Hence \(d_{r+1}(p)\) is not unimodal for \(31\le r\le47\), equivalently
\[
        32\le k\le 48.
\]

\subsection{A record-gap certificate for \texorpdfstring{\(41\le k\le8{,}600{,}001\)}{41 <= k <= 8,600,001}}
\label{sec:finite-large}

We use the two prime-record certificates from Section \ref{sec:inputs}.

We first show that the twin gap gives a strict ascent for every
\[
        40\le r\le8{,}600{,}000.
\]
Since \(s_T\) has \(71298\) decimal digits,
\[
        10^{71297}\le s_T<10^{71298}.
\]
Because \(\eps(y)\) is decreasing for \(y>1\), by \eqref{eq:A-upper}, \eqref{eq:B-interval}, and \eqref{eq:C-interval},
\[
\begin{aligned}
        A(s_T)
        &<
        \log\log(10^{71298})
        +0.261497212847643
        +\eps(10^{71297})
        +0.773157136700943  \\
        &<13.04331036.
\end{aligned}
\]
Therefore, at the gap \(s_T\to s_T+2\),
\[
        R_r(s_T^-)
        \ge
        \frac r{A(s_T)}
        \ge
        \frac{40}{13.04331036}
        >3.066>3=2+1.
\]
It remains to note that \(R_r(s_T^-)\) is defined throughout this range. Applying \eqref{eq:nth-prime-upper} with \(n=8{,}600{,}000\) gives
\[
        p_{8{,}600{,}000}<152{,}960{,}215<s_T.
\]
Thus at least \(8{,}600{,}000\) primes lie below \(s_T\). By Lemma \ref{lem:threshold}, there is a strict ascent at \(s_T\to s_T+2\) for every \(40\le r\le8{,}600{,}000\).

Now we show that the large gap gives a strict descent for all
\[
        40\le r\le8{,}600{,}000.
\]
Since \(s_L\) has \(18662\) decimal digits,
\[
        s_L>10^{18661}.
\]
Using \eqref{eq:A-lower-sharp}, we get
\[
\begin{aligned}
        A(s_L^-)
        &\ge
        A(10^{18661}) \\
        &>
        \log\log(10^{18661})
        +0.261497212847642
        -\eps(10^{18661})
        +0.773156636699192
        -\frac1{10^{18661}-1}\\
        &>11.70287735.
\end{aligned}
\]
For \(r\le8{,}600{,}000\),
\[
        W_{r-1}\le W_{8{,}599{,}999}=A(p_{8{,}599{,}999}).
\]
By \eqref{eq:nth-prime-upper},
\[
        p_{8{,}599{,}999}\le U,
\]
where
\[
        U:=8{,}599{,}999\left(
        \log(8{,}599{,}999)+\log\log(8{,}599{,}999)-1+
        \frac{\log\log(8{,}599{,}999)-2}{\log(8{,}599{,}999)}
        \right).
\]
Outward-rounded interval arithmetic gives
\[
        U<152{,}960{,}196.
\]
Then \eqref{eq:A-upper}, \eqref{eq:B-interval}, and \eqref{eq:C-interval} give
\[
        W_{r-1}\le A(152{,}960{,}196)<3.9713.
\]
In particular,
\[
        A(s_L^-)>W_{r-1},
\]
so there are at least \(r\) primes below \(s_L\), and Lemma \ref{lem:Rbounds} applies at \(s_L^-\).  Consequently, at the gap \(s_L\to s_L+g_L\),
\[
\begin{aligned}
        R_r(s_L^-)
        &\le
        \frac{r}{A(s_L^-)-W_{r-1}}\\
        &<
        \frac{8{,}600{,}000}{11.70287735-3.9713}\\
        &<1{,}112{,}322
        <1{,}113{,}107
        =g_L+1.
\end{aligned}
\]
Thus Lemma \ref{lem:threshold} gives a strict descent at \(s_L\to s_L+g_L\).

Finally, the strict ascent at \(s_T\to s_T+2\) occurs after this strict descent, because \(s_T\) has \(71298\) decimal digits while both endpoints of the large gap have \(18662\) digits.  Therefore \(d_{r+1}(p)\) is not unimodal for
\[
        40\le r\le8{,}600{,}000,
\]
equivalently
\[
        41\le k\le8{,}600{,}001.
\]

Combining the three finite parts gives the following proposition.

\begin{proposition}\label{prop:finite}
For every integer
\[
        4\le k\le8{,}600{,}001,
\]
the sequence \(p\mapsto d_k(p)\) is not unimodal.
\end{proposition}

\section{The Uniform Tail: \texorpdfstring{\(k\ge8{,}600{,}002\)}{k >= 8,600,002}}
\label{sec:tail}

We now prove all remaining cases by one uniform argument.  Put
\[
        r:=k-1.
\]
Thus it suffices to prove the claim for every integer
\[
        r\ge8{,}600{,}001.
\]
Set
\[
        v:=\log r,
        \qquad
        x:=0.99\,\frac r{\log r}=0.99\,\frac r v.
\]
Since
\[
        \log(8{,}600{,}001)>15.96,
\]
we have
\begin{equation}\label{eq:v-lower}
v>15.96.
\end{equation}
The function \(r/\log r\) is increasing for \(r>e\), and hence
\begin{equation}\label{eq:x-lower}
x\ge 0.99\frac{8{,}600{,}001}{\log(8{,}600{,}001)}
        >533{,}000.
\end{equation}
Consequently
\begin{equation}\label{eq:x-half-lower}
\frac x2>266{,}500>3275.
\end{equation}
Outward-rounded interval arithmetic gives, for every \(x\) in this range,
\begin{equation}\label{eq:logx-error}
\frac1{2\log^2 x}<0.002876,
\end{equation}
\begin{equation}\label{eq:logxhalf-error}
\frac1{2\log^2(x/2)}<0.00321,
\end{equation}
\begin{equation}\label{eq:theta-error-small}
\frac{1.2323}{\log x}<0.1,
\end{equation}
and
\begin{equation}\label{eq:M-factor-bound}
\left(1+\frac1{2\log^2x}\right)
        \left(1+\frac1{36260}\right)
        +\frac{\log 8}{x}
        <1.003.
\end{equation}
Moreover, if \(\log y>18.9\), then
\begin{equation}\label{eq:eps-small}
\eps(y)<0.001.
\end{equation}

We shall also use the following elementary inequalities.  First,
\begin{equation}\label{eq:tail-elementary-one}
0.44v-2\log v-1.300>0
\end{equation}
for every \(v\ge15.96\).  Indeed, the derivative is \(0.44-2/v>0\) on this interval, and the value at \(15.96\) is \(>0.18\).  Second, since \(v>15.96>1.50\),
\begin{equation}\label{eq:tail-elementary-two}
\frac{v}{0.99(v-1.50)}>\frac1{0.99}>1.010.
\end{equation}

\subsection{A large gap before \texorpdfstring{\(4P\)}{4P}}

By \eqref{eq:short-interval-prime}, choose a prime
\[
        q\in\left(x,\ x\left(1+\frac1{2\log^2x}\right)\right].
\]
Let \(q^-\) be the prime immediately preceding \(q\), and put
\[
        P:=\prod_{p\le q}p.
\]
Then
\[
        \log P=\vartheta(q).
\]
By \eqref{eq:theta-lower}, \(q>x\), and \eqref{eq:theta-error-small},
\begin{equation}\label{eq:logP-lower}
\log P
        =
        \vartheta(q)
        >
        q\left(1-\frac{1.2323}{\log q}\right)
        >
        0.9x.
\end{equation}
By \eqref{eq:theta-upper}, \eqref{eq:logx-error}, and \eqref{eq:M-factor-bound},
\begin{equation}\label{eq:log8P-upper}
\begin{aligned}
        \log(8P)
        &=
        \vartheta(q)+\log 8\\
        &<
        q\left(1+\frac1{36260}\right)+\log8\\
        &\le
        x\left(1+\frac1{2\log^2x}\right)
        \left(1+\frac1{36260}\right)+\log8\\
        &<1.003x.
\end{aligned}
\end{equation}

We now construct a block of consecutive composite integers.  For each prime \(p\le q\), define a residue class \(a_p\) by
\[
        a_p\equiv
        \begin{cases}
        q^- \pmod p,&p<q^-,\\
        q^- -1\pmod {q^-},&p=q^-,\\
        q^- +1\pmod q,&p=q.
        \end{cases}
\]
For every integer
\[
        1\le m\le2q^- -1,
\]
there exists a prime \(p\le q\) such that
\[
        m\equiv a_p\pmod p.
\]
Indeed:
\begin{itemize}
\item if \(1\le m\le q^- -2\), choose a prime divisor \(p\) of \(q^- -m\);
\item if \(m=q^- -1\), choose \(p=q^-\);
\item if \(m=q^-\), choose \(p=2\);
\item if \(m=q^-+1\), choose \(p=q\);
\item if \(q^-+2\le m\le2q^- -1\), choose a prime divisor \(p\) of \(m-q^-\).
\end{itemize}
Each case gives \(m\equiv a_p\pmod p\).

By the Chinese remainder theorem, choose \(a\) modulo \(P\) such that
\[
        a\equiv -a_p\pmod p
        \qquad(p\le q),
\]
and choose the representative \(0\le a<P\).  Then for every
\[
        1\le m\le2q^- -1
\]
there is a prime \(p\le q\) such that
\[
        a+2P+m\equiv a+m\equiv0\pmod p.
\]
Also
\[
        a+2P+m>2P>q\ge p,
\]
so \(a+2P+m\) is composite.  Hence
\[
        a+2P+1,\ a+2P+2,\ \ldots,\ a+2P+2q^- -1
\]
is a block of consecutive composite integers.

Since \(P\) contains the factors \(2,q^-,q\), and \(q>q^-\), we have
\[
        P\ge2q^-q>2q^-.
\]
Thus the above block is contained in \((2P,4P)\).  Therefore there is a prime gap \(G_-\) surrounding this block and satisfying
\begin{equation}\label{eq:Gminus-block}
G_-\ge2q^-.
\end{equation}

It remains to lower-bound \(q^-\).  By \eqref{eq:short-interval-prime}, applied to \(x/2\), there is a prime
\[
        s\in
        \left(\frac x2,\ \frac x2\left(1+\frac1{2\log^2(x/2)}\right)\right].
\]
By \eqref{eq:logxhalf-error}, this upper endpoint is \(<x\), while \(q>x\).  Therefore
\begin{equation}\label{eq:qminus-half}
q^-\ge s>\frac x2.
\end{equation}
Apply \eqref{eq:short-interval-prime} once more, now to \(q^-\).  Since \(q\) is the first prime after \(q^-\),
\[
        q\le q^-\left(1+\frac1{2\log^2q^-}\right).
\]
Together with \(q>x\) and \(q^->x/2\), this gives
\[
        q^-
        >
        \frac{x}{1+\frac1{2\log^2q^-}}
        >
        \frac{x}{1+\frac1{2\log^2(x/2)}}.
\]
Using \eqref{eq:logxhalf-error},
\[
        q^->\frac{x}{1.00321}.
\]
Therefore, by \eqref{eq:Gminus-block},
\begin{equation}\label{eq:Gminus-lower}
G_->\frac{2x}{1.00321}>1.993x.
\end{equation}

\subsection{A smaller later gap inside \texorpdfstring{\((4P,8P]\)}{(4P,8P]}}

Let
\[
        M:=\log(8P),
        \qquad
        N:=\pi(8P)-\pi(4P).
\]
By \eqref{eq:logP-lower}, \(M>20\).  By \eqref{eq:pi-lower} and \eqref{eq:pi-upper},
\[
        N
        \ge
        \frac{8P}{M-1}
        -
        \frac{4P}{M-\log2-1.1}.
\]
Set
\[
        D(M):=\frac8{M-1}-\frac4{M-\log2-1.1}.
\]
A direct calculation gives
\[
        D(M)-\frac4M
        =
        \frac{
        2\bigl((9-10\log2)M-(10\log2+11)\bigr)
        }
        {
        5M(M-1)(M-\log2-1.1)
        }.
\]
For \(M>20\),
\[
        9-10\log2>2,
        \qquad
        10\log2+11<18.
\]
Hence
\[
        (9-10\log2)M-(10\log2+11)>2M-18.
\]
Also
\[
        M(M-1)(M-\log2-1.1)<M^3.
\]
Thus
\[
        D(M)-\frac4M
        >
        \frac{2(2M-18)}{5M^3}
        >
        \frac8{M^3}.
\]
Therefore
\[
        D(M)>\frac4M+\frac8{M^3},
\]
and so
\[
        N>\frac{4P}{M}+\frac{8P}{M^3}.
\]
Since \(8P=\e^M\) and \(M>20\),
\[
        \frac{8P}{M^3}=\frac{\e^M}{M^3}>1.
\]
Consequently,
\begin{equation}\label{eq:N-lower}
N>\frac{4P}{M}+1.
\end{equation}

Let the primes in \((4P,8P]\) be
\[
        s_1<s_2<\cdots<s_N.
\]
Any consecutive pair \(s_\ell,s_{\ell+1}\) is also consecutive among all primes, since every prime between them would also lie in \((4P,8P]\).  Moreover
\[
        \sum_{\ell=1}^{N-1}(s_{\ell+1}-s_\ell)
        =
        s_N-s_1
        <4P.
\]
By \eqref{eq:N-lower}, some global prime gap \(G_+\) in \((4P,8P]\) satisfies
\[
        G_+<\frac{4P}{N-1}<M.
\]
Using \eqref{eq:log8P-upper},
\begin{equation}\label{eq:Gplus-upper}
G_+<M<1.003x.
\end{equation}
This gap \(G_+\) occurs after \(G_-\).  Indeed, \(G_-\) surrounds a composite block lying inside \((2P,4P)\), while \(G_+\) lies inside \((4P,8P]\).  If the right endpoint of \(G_-\) is the first prime in \((4P,8P]\), then \(G_+\) starts at that prime or later, and hence still occurs after \(G_-\).

\subsection{The large gap gives a strict descent}

Let \(G_-\) occur between the consecutive primes
\[
        p_i,\ p_{i+1}.
\]
We first show that
\begin{equation}\label{eq:piminus-P}
p_{i-1}>P.
\end{equation}
Let \(L:=\log P\).  By \eqref{eq:logP-lower}, \(L>20\).  From \eqref{eq:pi-lower} and \eqref{eq:pi-upper},
\[
\begin{aligned}
        \pi(2P)-\pi(P)
        &\ge
        \frac{2P}{L+\log2-1}
        -
        \frac{P}{L-1.1}\\
        &=
        P\left(
        \frac2{L+\log2-1}
        -
        \frac1{L-1.1}
        \right).
\end{aligned}
\]
For \(L>20\),
\[
        \frac2{L+\log2-1}
        -
        \frac1{L-1.1}
        >
        \frac1{2L}.
\]
Thus
\[
        \pi(2P)-\pi(P)>\frac{P}{2\log P}>2.
\]
Hence \((P,2P]\) contains at least two primes.  Since the block of composites surrounded by \(G_-\) starts after \(2P\), the prime immediately preceding the left endpoint of \(G_-\) is \(>P\).  This proves \eqref{eq:piminus-P}, and therefore
\begin{equation}\label{eq:A-piminus}
A(p_{i-1})\ge A(P).
\end{equation}
Once \eqref{eq:AminusW-lower} below is proved, \eqref{eq:A-piminus} also implies \(A(p_{i-1})>W_{r-1}\), hence \(i-1\ge r\) and Lemma \ref{lem:Rbounds} is applicable at the descent gap.

We next prove
\begin{equation}\label{eq:AminusW-lower}
A(P)-W_{r-1}>0.56\log r.
\end{equation}
By the definition of \(W_N\),
\[
        W_{r-1}=A(p_{r-1}).
\]
By \eqref{eq:nth-prime-upper}, applied to \(n=r-1\), with \(t:=\log(r-1)\),
\[
        p_{r-1}
        \le
        (r-1)\left(
        t+\log t-1+\frac{\log t-2}{t}
        \right).
\]
Since \(r\ge8{,}600{,}001\), we have \(t>15.96\).  Define
\[
        h(t):=0.2t-\log t+1-\frac{\log t-2}{t}.
\]
Then
\[
        h'(t)=0.2-\frac1t+\frac{\log t-3}{t^2}.
\]
For \(t\ge15.96\),
\[
        h'(t)>0.2-\frac1{15.96}-\frac{0.230}{15.96^2}>0,
\]
and
\[
        h(15.96)>1.37.
\]
Hence \(h(t)>0\), which is equivalent to
\[
        \log t-1+\frac{\log t-2}{t}<0.2t.
\]
Therefore
\[
        p_{r-1}<1.2(r-1)t<1.2r\log r.
\]
It follows that
\begin{equation}\label{eq:W-upper-by-A}
W_{r-1}=A(p_{r-1})
        \le
        A(1.2r\log r).
\end{equation}
Since
\[
        \log(1.2r\log r)>18.9,
\]
\eqref{eq:A-upper-weak} and \eqref{eq:eps-small} give
\begin{equation}\label{eq:W-log-upper}
W_{r-1}
        \le
        \log\log(1.2r\log r)+B+1.001.
\end{equation}
On the other hand, by \eqref{eq:logP-lower},
\[
        \log P>0.9x>18.9.
\]
Thus \eqref{eq:A-lower-weak} and \eqref{eq:eps-small} imply
\begin{equation}\label{eq:A-P-log-lower}
A(P)
        \ge
        \log\log P+B-0.001
        >
        \log(0.9x)+B-0.001.
\end{equation}
Subtracting \eqref{eq:W-log-upper} from \eqref{eq:A-P-log-lower},
\begin{equation}\label{eq:AminusW-log}
A(P)-W_{r-1}
        >
        \log(0.9x)-\log\log(1.2r\log r)-1.002.
\end{equation}
We now bound the second logarithm.  Since \(v=\log r\ge15.96\), the function
\[
        0.2v-\log(1.2v)
\]
is increasing and is \(>0.23\) at \(v=15.96\).  Therefore
\[
        \log(1.2v)<0.2v,
\]
and hence
\[
        \log(1.2r\log r)
        =
        v+\log(1.2v)
        <
        1.2v.
\]
Thus
\begin{equation}\label{eq:loglog-upper}
\log\log(1.2r\log r)<\log(1.2v).
\end{equation}
Since
\[
        x=0.99\,\frac r v,
\]
we have
\[
        \log(0.9x)=v-\log v+\log(0.9\cdot0.99).
\]
Using \eqref{eq:AminusW-log} and \eqref{eq:loglog-upper},
\[
\begin{aligned}
        A(P)-W_{r-1}
        &>
        v-2\log v+\log(0.9\cdot0.99)-\log1.2-1.002\\
        &>
        v-2\log v-1.300.
\end{aligned}
\]
By \eqref{eq:tail-elementary-one},
\[
        v-2\log v-1.300>0.56v.
\]
Therefore \eqref{eq:AminusW-lower} is proved.

Now Lemma \ref{lem:Rbounds}, together with \eqref{eq:A-piminus} and \eqref{eq:AminusW-lower}, gives
\[
\begin{aligned}
        R_r(i-1)
        &\le
        \frac r{A(p_{i-1})-W_{r-1}}\\
        &\le
        \frac r{A(P)-W_{r-1}}\\
        &<
        \frac r{0.56\log r}.
\end{aligned}
\]
Since \(x=0.99r/\log r\),
\begin{equation}\label{eq:R-descent-upper}
R_r(i-1)
        <
        \frac{x}{0.99\cdot0.56}
        <1.805x.
\end{equation}
By \eqref{eq:Gminus-lower},
\[
        G_->1.993x.
\]
Thus
\[
        R_r(i-1)<1.805x<1.993x<G_-+1.
\]
Lemma \ref{lem:threshold} gives
\begin{equation}\label{eq:strict-descent-tail}
d_{r+1}(p_{i+1})<d_{r+1}(p_i).
\end{equation}

\subsection{The later small gap gives a strict ascent}

Let \(G_+\) occur between the consecutive primes
\[
        p_j,\ p_{j+1}.
\]
Since \(G_+\) lies in \((4P,8P]\),
\[
        p_{j-1}<p_j\le8P.
\]
Also \(p_{j-1}>P\), because \((P,2P]\) contains at least two primes and \(p_j>4P\).  Hence, by \eqref{eq:AminusW-lower}, \(A(p_{j-1})>W_{r-1}\), and in particular \(j-1\ge r\).  Lemma \ref{lem:Rbounds} gives
\[
        R_r(j-1)
        \ge
        \frac r{A(p_{j-1})}.
\]
Since \(p_{j-1}\le8P\),
\[
        A(p_{j-1})\le A(8P).
\]
Therefore
\begin{equation}\label{eq:R-ascent-pre}
R_r(j-1)\ge\frac r{A(8P)}.
\end{equation}
By \eqref{eq:A-upper-weak}, \eqref{eq:eps-small}, and \(M=\log(8P)>18.9\),
\[
        A(8P)\le \log\log(8P)+B+1.001.
\]
By \eqref{eq:log8P-upper},
\[
        \log(8P)<1.003x.
\]
Since \(B<0.262\),
\begin{equation}\label{eq:A-8P-upper}
A(8P)
        <
        \log(1.003x)+B+1.001
        <
        \log x+1.266.
\end{equation}
Now
\[
        \log x=v-\log v+\log0.99.
\]
For \(v\ge15.96\),
\[
        -\log v+\log0.99+1.266<-1.50.
\]
Thus
\begin{equation}\label{eq:A-8P-v}
A(8P)<v-1.50.
\end{equation}
Combining \eqref{eq:R-ascent-pre} and \eqref{eq:A-8P-v},
\[
        R_r(j-1)>
        \frac r{v-1.50}.
\]
Since \(r=xv/0.99\),
\[
        R_r(j-1)
        >
        x\cdot \frac{v}{0.99(v-1.50)}.
\]
By \eqref{eq:tail-elementary-two},
\begin{equation}\label{eq:R-ascent-lower}
R_r(j-1)>1.010x.
\end{equation}
On the other hand, by \eqref{eq:Gplus-upper},
\[
        G_+<1.003x.
\]
Since \(x>533{,}000\), we have \(1<0.001x\), and so
\[
        G_+ +1<1.004x<1.010x<R_r(j-1).
\]
By Lemma \ref{lem:threshold},
\begin{equation}\label{eq:strict-ascent-tail}
d_{r+1}(p_{j+1})>d_{r+1}(p_j).
\end{equation}

Equations \eqref{eq:strict-descent-tail} and \eqref{eq:strict-ascent-tail} show that the sequence \(p\mapsto d_{r+1}(p)\) first has a strict descent and later has a strict ascent.  Such a sequence cannot be unimodal.  Therefore we have proved:

\begin{proposition}\label{prop:tail}
For every integer
\[
        k\ge8{,}600{,}002,
\]
the sequence \(p\mapsto d_k(p)\) is not unimodal.
\end{proposition}

\section{Conclusion}

Proposition \ref{prop:finite} proves the theorem for
\[
        4\le k\le8{,}600{,}001,
\]
and Proposition \ref{prop:tail} proves it for
\[
        k\ge8{,}600{,}002.
\]
Together these two propositions prove Theorem \ref{thm:main} for every \(k\ge4\).  Combining this with Cambie's theorem for \(1\le k\le3\) proves Corollary \ref{cor:classification}.

\section*{Acknowledgements}

We thank Samuele Marro and Marcello Politi for early conversations about using multi-agent systems to tackle difficult scientific problems, for their helpful input on the early design of the Multiscalar Fields System, and for following the progress of this project.

We are grateful to Xiaozheng Han and Yaoping Xie, Ph.D. students at the School of Mathematical Sciences, Peking University, for independently checking the proof.

The Multiscalar Fields System used AI models as part of its proof search and verification workflow, including OpenAI’s gpt-5.4-pro and gpt-5.5, and Anthropic’s Claude Opus 4.7.

\appendix
\section{Numerical Verification}\label{app:verification}

A companion numerical verifier is publicly available at
\url{https://github.com/multiscalar/results/blob/main/erdos-690/numerical_verifier.py}.

The verifier certifies the explicit finite numerical comparisons appearing in the proof, rather than to re-prove the external analytic estimates or the published prime-record inputs cited in Section \ref{sec:inputs}. Conditional on those cited inputs, the script mechanically checks the displayed numerical inequalities used in the argument.

More precisely, it verifies:
\begin{itemize}
\item the finite prime sums in Section \ref{sec:finite}, using exact rational arithmetic;
\item the elementary symmetric-polynomial values in Table \ref{tab:smallcases}, using exact rational arithmetic;
\item the interval for \(C=\sum_p \frac{1}{p(p-1)}\), using directed outward-rounded summation for the finite partial sum together with the exact tail bound from Certificate \ref{cert:constants};
\item the displayed logarithmic numerical comparisons, using rational interval arithmetic for logarithms based on the \(\operatorname{atanh}\) series with a rigorous remainder bound.
\end{itemize}

The verifier does not use binary floating-point arithmetic.  Finite prime sums and symmetric-polynomial values are computed exactly with rational arithmetic, while the enclosure for \(C\) and the logarithmic inequalities are handled by rigorous outward-rounded interval methods.

\end{document}